\documentclass[12pt,reqno]{amsart}
\usepackage{amssymb}
\usepackage{amsfonts}
\usepackage[hidelinks]{hyperref}
\usepackage{times}
\usepackage{latexsym}
\usepackage{amsmath}
\usepackage[mathscr]{eucal}
\usepackage{amsthm}

\DeclareMathOperator{\Trace}{Trace}
\DeclareMathOperator{\Tr}{Tr}
\DeclareMathOperator{\Ric}{Ric}
\makeatletter
\let\origsection\section
\renewcommand\section{\@ifstar{\starsection}{\nostarsection}}
\newcommand\nostarsection[1]
{\sectionprelude\origsection{#1}\sectionpostlude}
\newcommand\starsection[1]
{\sectionprelude\origsection*{#1}\sectionpostlude}
\newcommand\sectionprelude{  \vspace{1.2em}
}
\newcommand\sectionpostlude{  \vspace{0.3em}
}
\def\thm@space@setup{  \thm@preskip=2ex   \thm@postskip=\thm@preskip }
\makeatother

\theoremstyle{plain}

\newtheorem{algorithm}{Algorithm}[section]
\newtheorem{thm}{Thm}

\newtheorem{corollary}[algorithm]{Corollary}

\newtheorem{lemma}[algorithm]{Lemma}

\newtheorem{theorem} [algorithm] {Theorem}
\newtheorem{theoremlet}[thm]{Theorem}

\newtheorem{corollarylet}[thm]{Corollary}

\newtheorem{proposition}[algorithm]{Proposition}

\theoremstyle{definition}
\newtheorem*{acknowledgment}{Acknowledgment}
\newtheorem{definition}[algorithm]{Definition}

\newtheorem{remark}[algorithm]{Remark}
\newtheorem{example}[algorithm]{Example}

\newtheorem*{examplenonumt}{Example}

\newtheorem*{remarknonum}{Remark}
\newtheorem*{definitionnonum}{Definition}

\begin{document}
\title[Submanifold restrictions]{Restrictions on Submanifolds via Focal
Radius Bounds}
\author{Luis Guijarro}
\address{Department of Mathematics, Universidad Aut\'{o}noma de Madrid, and
ICMAT CSIC-UAM-UCM-UC3M }
\email{luis.guijarro@uam.es}
\urladdr{http://verso.mat.uam.es/~luis.guijarro/}
\thanks{The first author was supported by research grants MTM2011-22612,
MTM2014-57769-3-P, and MTM2017- 85934-C3-2-P from the MINECO, and by ICMAT
Severo Ochoa project SEV-2015-0554 (MINECO)}
\author{Frederick Wilhelm}
\address{Department of Mathematics\\
University of California\\
Riverside, CA 92521}
\email{fred@math.ucr.edu}
\urladdr{https://sites.google.com/site/frederickhwilhelmjr/home}
\thanks{This work was supported by a grant from the Simons Foundation
(\#358068, Frederick Wilhelm)}
\date{\today }

\begin{abstract}
We give an optimal estimate for the norm of any submanifold's second
fundamental form in terms of its focal radius and the lower sectional
curvature bound of the ambient manifold.


This is a special case of a similar theorem for intermediate Ricci
curvature, and leads to a $C^{1,\alpha}$ compactness result for
submanifolds, as well as a \textquotedblleft soul-type\textquotedblright\
structure theorem for manifolds with nonnegative $k^{th}$--intermediate
Ricci curvature that have a closed submanifold with dimension $\geq k$ and
infinite focal radius.

To prove these results, we use a new comparison lemma for Jacobi fields from 
\cite{GuijWilh} that exploits Wilking's transverse Jacobi equation. The new
comparison lemma also yields new information about group actions, Riemannian
submersions, and submetries, including generalizations to intermediate Ricci
curvature of results of Chen and Grove. None of these results can be
obtained with just classical Riccati comparison (see Subsection \ref%
{subsect: classic comp} for details.)
\end{abstract}

\maketitle

\pdfbookmark[1]{Introduction}{Introduction}


Submanifolds restrict the Riemannian geometry of the space in which they
lie, but only if they satisfy extra conditions. One constraint comes from
the tubular neighborhood theorem. It asserts that given any compact
submanifold $S$, there is a positive $r_{0}$ such that the normal disc
bundle $D_{r_{0}}(S)$ is diffeomorphic to an open neighborhood of $S$; the
diffeomorphism can be realized via the normal exponential map of $S$. This
motivates the notion of \emph{focal radius, }which is the maximum $r_{0}$
such that the normal exponential map is a local diffeomorphism of $%
D_{r_{0}}(S)$.


Our first result shows that we can  bound the norm of the second fundamental
form of any submanifold in terms of its focal radius and the ambient
manifold's lower curvature bound.

\begin{theoremlet}
\label{II estimate}
For $\kappa =-1,0,$ or $1,$ let $M$ be a complete
Riemannian $n$--manifold with sectional curvature $\geq \kappa ,$ and let $N$
be any submanifold of $M$ with $\dim \left( N\right) \geq 1.$ Then the
second fundamental form $\mathrm{II}_{N}$ of $N$ satisfies%
\begin{equation}
\begin{array}{ll}
\left\vert\, \mathrm{II}_{N}\right\vert \leq \cot \left( \mathrm{FocalRadius}%
\left( N\right) \right) & \text{if }\kappa =1,\vspace{0.07in} \\ 
\left\vert \,\mathrm{II}_{N}\right\vert \leq \dfrac{1}{\mathrm{FocalRadius}%
\left( N\right) } & \text{if }\kappa =0,\text{ and}\vspace{0.07in} \\ 
\left\vert \,\mathrm{II}_{N}\right\vert \leq \coth \left( \mathrm{FocalRadius%
}\left( N\right) \right) & \text{if }\kappa =-1.%
\end{array}
\label{II bound inequal}
\end{equation}

In particular, if $\kappa =0$ and the focal radius of $N$ is infinite, then $%
N$ is totally geodesic.
\end{theoremlet}

We emphasize that $N$ does not need to be closed or even complete. On the
other hand, if $M$ happens to be closed, the presence of a lower curvature
bound $\kappa $ is automatic, and after rescaling, we can take $\kappa $ to
be $-1,0,$ or $1.$ So for closed manifolds, Theorem \ref{II estimate} is
universal in the sense that it applies to any submanifold of any Riemannian
manifold. The upper bound is, moreover, optimal. Metric balls in space forms
show that for every $\kappa $ and every possible focal radius, there is a
hypersurface in a space with constant curvature $\kappa $ for which
Inequality (\ref{II bound inequal}) is an equality.


As a consequence of this result, we show that submanifolds with focal radius
bounded from below and diameter bounded from above have only finitely many
diffeomorphism types, a result that is of independent interest

\begin{theoremlet}
\label{finiteness thm}Let $M$ be a compact Riemannian manifold. Given $D,r>0$
the class $\mathcal{S}$ of closed Riemannian manifolds that can be
isometrically embedded into $M$ with focal radius $\geq r$ and intrinsic
diameter $\leq D$ is precompact in the $C^{1,\alpha }$--topology. In
particular, $\mathcal{S}$ contains only finitely many diffeomorphism types.
\end{theoremlet}

Theorem \ref{finiteness thm} is optimal in the sense that neither the
hypothesis on the focal radius nor the hypothesis on the diameter can be
removed. If either hypothesis is removed, then, after rescaling, all
Riemannian $k$--manifolds can occur in a flat $n$--torus, provided $n>>k$
(see example \ref{focal ex}).

Since Inequality \eqref{II bound inequal} applies to any submanifold of any
Riemannian manifold, the Gauss Equation implies that the class $\mathcal{S}$
in Theorem \ref{finiteness thm}, has uniformly bounded sectional curvature.
Theorem \ref{finiteness thm} follows from this and Cheeger's Finiteness
Theorem (\cite{Cheeg2}), provided the class $\mathcal{S}$ also has a uniform
lower bound for its volume. We achieve the lower volume bound as a
consequence of Heintze and Karcher's tube formula (\cite{HeKar}, see Lemma %
\ref{bnd geom lemma}, below).

Since Theorem \ref{II estimate} gives a new way to estimate curvature, it has many
corollaries. For example, using again the Gauss Equation, Theorem \ref{II estimate} provides
us with a simple proof of the following two statements, that are valid for
submanifolds of arbitrary codimension.

\begin{corollarylet}
\leavevmode\vspace{-0.4\baselineskip} 

\begin{itemize}
\item A submanifold of $\mathbb{S}^{n}$ is positively curved if its focal
radius is $>\frac{\pi }{4}.$

\item A submanifold of any hyperbolic manifold is nonpositively curved if it
has infinite focal radius.
\end{itemize}

\end{corollarylet}

The Clifford torus in $\mathbb{S}^{3}$ has focal radius $\frac{\pi }{4},$ so
the first statement of the corollary is optimal.

Theorem \ref{II estimate} is obtained as a consequence of a more general
bound on the second fundamental form of a submanifold, that is true in the
more general context of bounds on the intermediate Ricci curvature (see
Theorem \ref{shape est. thm} below). 
As a consequence, we recapture all of the rigidity of the Soul Theorem (\cite%
{CheegGrom}, \cite{Guij2},\cite{Perl}, \cite{Shar}, \cite{Wilk1}), provided
a manifold with $\Ric_{k}$ $\geq 0$ contains a closed submanifold with
infinite focal radius. (See \cite{GuijWilh} or \cite{Wu} for the definition
of intermediate Ricci curvature.)

\begin{theoremlet}
\label{inter ricci soul them}Let $M$ be a complete Riemannian $n$--manifold
with $\Ric_{k}$ $\geq 0,$ and let $N$ be any closed submanifold of $M$ with $%
\dim \left( N\right) \geq k$ and infinite focal radius. Then:$\vspace{0.1in}$

\noindent 1. $N$ is totally geodesic.$\vspace{0.1in}$

\noindent 2. The normal bundle $\nu \left( N\right) $ with the pull back
metric $\left( \exp _{N}^{\perp }\right) ^{\ast }\left( g\right) $ is a
complete manifold with $\Ric_{k}$ $\geq 0.\vspace{0.1in}$

\noindent 3. $\exp _{N}^{\perp }:\left( \nu \left( N\right) ,\left( \exp
_{N}^{\perp }\right) ^{\ast }g\right) \longrightarrow \left( M,g\right) $ is
a Riemannian cover.$\vspace{0.1in}$

\noindent 4. The zero section $N_{0}$ is totally geodesic in $\left( \nu
\left( N\right) ,\left( \exp _{N}^{\perp }\right) ^{\ast }\left( g\right)
\right) .\vspace{0.1in}$

\noindent 5. The projection $\pi :\left( \nu \left( N\right) ,\left( \exp
_{N}^{\perp }\right) ^{\ast }\left( g\right) \right) \longrightarrow N$ is a
Riemannian submersion.$\vspace{0.1in}$

\noindent 6. If $c:I\longrightarrow N$ is a unit speed geodesic in $N,$ and $%
V$ is a parallel normal field along $c,$ then 
\begin{equation*}
\Phi :I\times \mathbb{R}\longrightarrow M,\text{\hspace{0.6in}}\Phi \left(
s,t\right) =\exp _{c\left( s\right) }^{\perp }\left( tV\left( s\right)
\right)
\end{equation*}%
is a totally geodesic immersion whose image has constant curvature $0.%
\vspace{0.1in}$

\noindent 7. All radial sectional curvatures from $N$ are nonnegative. That
is, for $\gamma \left( t\right) =\exp _{N}^{\perp }\left( tv\right) $ with $%
v\in \nu \left( N\right) ,$ the curvature of any plane containing $\gamma
^{\prime }\left( t\right) $ is nonnegative.$\vspace{0.1in}$

\noindent 8. If $n\geq 3$ and $k\leq n-2,$ then for all $r>0,$ the intrinsic
metric on $\exp _{N}^{\perp }\left( S\left( N_{0},r\right) \right) $ has $%
\Ric_{k}\geq 0,$ where $S\left( N_{0},r\right) $ is the metric $r$--sphere
around the zero section $N_{0}$ in $\nu \left( N\right) .$
\end{theoremlet}

The version of Part 8 of Theorem \ref{inter ricci soul them} for nonnegative
sectional curvature and small $r$ is similar to Theorem 2.5 of \cite%
{GuijWals}. In the latter result, $N$ needs to be a soul of $M$ but can have
any focal radius.

In the case of Ricci curvature, Theorem \ref{inter ricci soul them} is
Theorem 3 of \cite{EschOSu}, but in the sectional curvature case, it yields
new information about open nonnegatively curved manifolds.

\begin{corollarylet}
\label{soul info cor}Let $N$ be a closed submanifold in a complete,
noncompact, simply connected nonnegatively curved manifold $\left(
M,g\right) .$ If $N$ has infinite focal radius, then $N$ is a soul of $M.$
\end{corollarylet}

While examples show that souls need not have infinite focal radius, using
the main theorem of \cite{Guij}, we can always modify the metric of $M$ so
that its soul has infinite focal radius.

Theorem \ref{inter ricci soul them} also imposes rigidity on compact
nonnegatively curved manifolds that contain closed submanifolds with no
focal points (see Corollary \ref{inf focal nonneg cor}).

To prove Theorems \ref{II estimate} and \ref{inter ricci soul them} we use
the new Jacobi field comparison lemma from \cite{GuijWilh}. It also has
consequences for Riemannian submersions, isometric group actions, and
Riemannian foliations of manifolds with positive intermediate Ricci
curvature. To state them succinctly, we recall the definition of
\textquotedblleft manifold submetry\textquotedblright\ from \cite{ChenGrv}.

\begin{definitionnonum}
A submetry 
\begin{equation*}
\pi :M\longrightarrow X
\end{equation*}%
of a Riemannian manifold is called a \textquotedblleft manifold
submetry\textquotedblright\ if and only if $\pi ^{-1}\left( x\right) $ is a
closed smooth submanifold for all $x\in X$ and every geodesic of $M$ that is
initially perpendicular to a fiber of $\pi $ is everywhere perpendicular to
the fibers of $\pi .$
\end{definitionnonum}

If the leaves of a singular Riemannian foliation are closed, then as pointed
out in \cite{ChenGrv}, its quotient map is a manifold submetry. Thus the
following result applies to singular Riemannian foliations with closed
leaves. In particular, it applies to quotient maps of isometric group
actions and to Riemannian submersions. In it, we use the term
\textquotedblleft geodesic\textquotedblright\ to mean a curve that locally
minimizes distances but need not be a globally shortest path.

\begin{theoremlet}
\label{submetry cor}Let $\pi :M\longrightarrow X$ 
be a manifold submetry of a complete Riemannian $n$--manifold with $\Ric%
_{k}\geq k.$ Suppose that for some $x\in X,$ $\dim \pi ^{-1}\left( x\right)
\geq k.$

\begin{enumerate}
\item For every geodesic $\gamma $ emanating from $x$, either $\gamma $ does
not extend to a geodesic on any interval that properly contains $\left[ 
\frac{-\pi }{2},\frac{\pi }{2}\right] ,$ or $\gamma $ has a conjugate point
to $x$ in $\left[ \frac{-\pi }{2},\frac{\pi }{2}\right] .$ In particular, if 
$X$ is smooth and $\pi $ is a Riemannian submersion, then the conjugate
radius of $X$ at $x$ is $\leq \frac{\pi }{2}.$

\smallskip

\item If all geodesics emanating from $x$ extend to geodesics on $\left[ 
\frac{-\pi }{2},\frac{\pi }{2}\right] $ and are free of conjugate points on $%
\left( \frac{-\pi }{2},\frac{\pi }{2}\right) ,$ then $\pi ^{-1}\left(
x\right) $ is totally geodesic in $M,$ and the universal cover of $M$ is
isometric to the sphere or a projective space with the standard metrics.

\smallskip

\item If $\dim \pi ^{-1}\left( x\right) \geq k$ for some $x\in X$ for which $%
\max \left\{ \mathrm{dist}_{x}\right\} =\mathrm{diam}\left( X\right) , $
then the diameter of $X$ is $\leq \frac{\pi }{2}.$
\end{enumerate}



\end{theoremlet}

The relevant definition of conjugate points in length spaces is given in \ref%
{variat conj dfn}.

Projective spaces viewed as the bases of Hopf fibrations show that the
conjugate radius estimate in Part 1 is optimal. The conclusion about the
extendability of $\gamma $ is also optimal.

\begin{examplenonumt}
Let $SO\left( n\right) $ act reducibly on the unit sphere, $\mathbb{S}^{n},$
in the usual way, by cohomogeneity one. Let $x\in \mathbb{S}^{n}/SO\left(
n\right) $ be the orbit of the equator. The geodesic passing through $x$ at
time $0$ extends to $\left[ -\frac{\pi }{2},\frac{\pi }{2}\right] ,$ where
it is free of conjugate points, but it does not extend to any larger
interval.
\end{examplenonumt}

This example also shows that for Part 3 of Theorem $\ref{submetry cor},$ it
is not enough to know that $\dim \pi ^{-1}\left( x\right) \geq k$ for \emph{%
some} $x\in X;$ we must also assume that $x$ realizes the diameter of $X.$

The remainder of the paper is organized as follows. In Section \ref%
{notations Sections} we establish notations and conventions. In Section \ref%
{comp lemma}, we review the comparison lemma and focal radius theorems of 
\cite{GuijWilh}. Theorems \ref{II estimate} and \ref{inter ricci soul them}
are proven in Section \ref{focal and lower sect}. In Section \ref{submanf
finitieness secti}, we prove Theorem \ref{finiteness thm}, provide examples
that show it is optimal, and state another finiteness theorem whose proof is
essentially the same as the proof of Theorem \ref{finiteness thm}. The paper
concludes with Section \ref{subm and conj sect} where we prove Theorem \ref%
{submetry cor} and state some of its corollaries for isometric groups
actions.

\begin{remarknonum}
To keep the exposition simple, we have stated all of our results with the
global hypothesis $\Ric_{k}$ $M\geq k\cdot \kappa ;$ however, most of them
also hold with only the corresponding hypothesis about radial intermediate
Ricci curvatures. That is, for any geodesic $\gamma $ that leaves our
submanifold orthogonally at time $0,$ we only need 
\begin{equation*}
\displaystyle\sum\limits_{i=1}^{k}\mathrm{sec}\left( \dot{\gamma}%
,E_{i}\right) \geq k\cdot \kappa
\end{equation*}%
for any orthonormal set $\left\{ \dot{\gamma},E_{1},\ldots ,E_{k}\right\} .$
This remark applies to Theorems \ref{II estimate}, \ref{inter ricci soul
them}, and $\ref{submetry cor}$, except for Part 2 of Theorem $\ref{submetry
cor}$ for which our proof still requires the global hypothesis.
\end{remarknonum}

\begin{acknowledgment}
We are grateful to Xiaoyang Chen for raising the question of whether Theorem
A in \cite{ChenGrv} might hold for Ricci curvature, to Ravi Shankar and
Christina Sormani for correspondence about conjugate points in length
spaces, to Pedro Sol\'{o}rzano for conversations about submetries and
holonomy, and to Karsten Grove for useful comments. Special thanks go to
Universidad Autonoma de Madrid for hosting a stay by the second author
during which this work was initiated.
\end{acknowledgment}

\section{Notations and Conventions\label{notations Sections}}

Let $\gamma :\left( -\infty ,\infty \right) \longrightarrow M$ be a unit
speed geodesic in a complete Riemannian $n$--manifold $M.$ Call an $\left(
n-1\right) $--dimensional subspace $\Lambda $ of normal Jacobi fields along $%
\gamma ,$ \emph{Lagrangian, }if the restriction of the Riccati operator to $%
\Lambda $ is self adjoint, that is, if 
\begin{equation*}
\left\langle J_{1}\left( t\right) ,J_{2}^{\prime }\left( t\right)
\right\rangle =\left\langle J_{1}^{\prime }\left( t\right) ,J_{2}\left(
t\right) \right\rangle
\end{equation*}%
for all $t$ and for all $J_{1},J_{2}\in \Lambda $.

For a subspace $V\subset \Lambda $ we write 
\begin{equation}
V\left( t\right) \equiv \left\{ \left. J\left( t\right) \text{ }\right\vert 
\text{ }J\in V\right\} \oplus \left\{ \left. J^{\prime }\left( t\right) 
\text{ }\right\vert \text{ }J\in V\text{ and }J\left( t\right) =0\,\right\} .
\label{dfn of eval eqn}
\end{equation}

Given a submanifold $N$ of the Riemannian manifold $M,$ we let $\nu \left(
N\right) $ be its normal bundle$.$ We use $\pi $ for the projection of $\nu
\left( N\right) $ onto $N$, and $N_{0}$ for the $0$--section of $\nu \left(
N\right) .$ If $\gamma $ is a geodesic with $\gamma ^{\prime }\left(
0\right) \perp N,$ we consider variations of $\gamma $ by geodesics that
leave $N$ orthogonally at time $0.$ We let $\Lambda _{N}$ be the the
corresponding variations fields; note that $\Lambda _{N}$ is Lagrangian.
Lemma 4.1 on page 227 of \cite{doCarm} says that $\Lambda _{N}$ is the set
of normal Jacobi fields $J$ given by:%
\begin{equation}
\Lambda _{N}\equiv \left\{ J|J\left( 0\right) =0\text{, }J^{\prime }\left(
0\right) \in \nu _{\gamma \left( 0\right) }\left( N\right) \right\} \oplus
\left\{ J|J\left( 0\right) \in T_{\gamma \left( 0\right) }N\text{ and }{%
J^{\prime }}\left( 0\right) =\mathrm{S}_{\gamma ^{\prime }\left( 0\right)
}J\left( 0\right) \right\} ,  \label{dfn of Lambda_N}
\end{equation}%
where $\mathrm{S}_{\gamma ^{\prime }\left( 0\right) }$ is the shape operator
of $N$ in the direction of $\gamma ^{\prime }\left( 0\right) ,$ that is, 
\begin{eqnarray*}
\mathrm{S}_{\gamma ^{\prime }\left( 0\right) } &:&T_{\gamma \left( 0\right)
}N\longrightarrow T_{\gamma \left( 0\right) }N\text{ is} \\
\mathrm{S}_{\gamma ^{\prime }\left( 0\right) } &:&w\longmapsto \left( \nabla
_{w}\gamma ^{\prime }\left( 0\right) \right) ^{TN}.
\end{eqnarray*}

For every $t\in \mathbb{R}$, we let $\mathcal{E}_{t}:\Lambda \longrightarrow
T_{\gamma \left( t\right) }M$, be the evaluation map $\mathcal{E}_{t}\left(
J\right) =J\left( t\right) $. Unless otherwise indicated, we suppose that $%
\mathcal{E}_{t}$ is injective on $\left( t_{0},t_{\max }\right) .$ When this
occurs, we say that $\Lambda $ is nonsingular on $\left( t_{0},t_{\max
}\right) .$

Geodesics are parameterized by arc length, except if we say otherwise. $%
\gamma _{v}$ will be the unique geodesic tangent to $v$ at time $0.$

Finally, we use $\mathrm{sec}$ to denote sectional curvature.

\section{The Comparison Lemmas and their Consequences\label{comp lemma}}

To prove Theorems \ref{II estimate} and \ref{inter ricci soul them} we
exploit the new Jacobi field comparison lemmas from \cite{GuijWilh}. We
review these here, and refer the reader to \cite{GuijWilh} for a full
exposition. 

Lagrangian subspaces in $2$--dimensional constant curvature spaces are
spanned by single Jacobi fields of the form $\tilde{f}E,$ where $E$ is a
parallel field. After rescaling the metric, $\tilde{f}$ is one of the
following 
\begin{equation}
\tilde{f}\left( t\right) =\left\{ 
\begin{array}{ll}
\left( c_{1}\sin t+c_{2}\cos t\right) & \text{if }\kappa =1,\vspace{0.05in}
\\ 
\left( c_{1}t+c_{2}\right) & \text{if }\kappa =0,\vspace{0.05in} \\ 
\left( c_{1}\sinh t+c_{2}\cosh t\right) & \text{if }\kappa =-1,%
\end{array}%
\right.  \label{model Jacobi}
\end{equation}%
for a choice of $c_{1},c_{2}\in \mathbb{R}$.

For a subspace $W\subset \Lambda $, write 
\begin{equation*}
W\left( t\right) =\left\{ \left. J\left( t\right) \text{ }\right\vert \text{ 
}J\in W\right\} \oplus \left\{ \left. J^{\prime }\left( t\right) \text{ }%
\right\vert \text{ }J\in W\text{ and }J\left( t\right) =0\right\} ,
\end{equation*}%
and 
\begin{equation*}
P_{W,t}:\Lambda \left( t\right) \longrightarrow W\left( t\right)
\end{equation*}%
for orthogonal projection. If $S$ is the Riccati operator associated to $%
\Lambda ,$ then to abbreviate we write 
\begin{equation*}
\Trace S_t |_{W} \text{ for }\mathrm{Trace}\left( P_{W,t}\circ S_t
|_{W}\right) .
\end{equation*}
Finally, recall that a subspace $\mathcal{V}$ of $\Lambda$ has full index on
an interval $I$ if it contains any Jacobi field in $\Lambda$ that vanishes
at some point of $I$.

We can now state the comparison lemmas from \cite{GuijWilh} that we use here.


\begin{lemma}[Intermediate Ricci Comparison]
\label{sing soon Ric_k Lemma} For $\kappa =-1,0,$ or $1,$ let $\gamma
:\left( -\infty ,\infty \right) \longrightarrow M$ be a unit speed geodesic
in a complete Riemannian $n$--manifold $M$ with $\newline
Ric_{k}$ $\geq k\cdot \kappa .$ Let $\tilde{\lambda}_{\kappa }:\left[
t_{0},t_{\max }\right) \longrightarrow \mathbb{R}$ be any solution of the
scalar Riccati equation 
\begin{equation}
\tilde{\lambda}_{\kappa }^{\prime }+\tilde{\lambda}_{\kappa }^{2}+\kappa =0%
\text{.}  \label{const curv ODE}
\end{equation}%
Let $\Lambda $ be a Lagrangian subspace of normal Jacobi fields along $%
\gamma $ with Riccati operator $S$, and let $W_{t_{0}}\perp \gamma ^{\prime
}(t_{0})$ be some $k$--dimensional subspace such that 
\begin{equation}
\Trace S_{t_{0}}|_{W_{t_{0}}}\leq k\cdot \tilde{\lambda}_{\kappa }\left(
t_{0}\right) .  \label{small initial}
\end{equation}%
Denote by $\mathcal{V}$ the subspace of $\Lambda $ formed by those Jacobi
fields that are orthogonal to $W_{t_{0}}$ at $t_{0}$ and by $H(t)$ the
subspace of $\gamma ^{\prime }(t)^{\perp }$ that is orthogonal to $\mathcal{V%
}(t)$ at each $t\in (t_{0},t_{max})$. 
Assume that $\mathcal{V}$ is of full index in the interval $[t_{0},t_{max})$.

Then for all $t\in \left[ t_{0},t_{\max }\right)$, 
\begin{equation}
\Trace S_t|_{H(t)} \leq k\cdot \tilde{\lambda}_{\kappa }\left( t\right) .
\label{small future Inequal}
\end{equation}

Moreover, if $\lim_{t\rightarrow t_{\max }^{-}}\tilde{\lambda}_{\kappa
}\left( t\right) =-\infty $ then the Jacobi equation splits orthogonally
along $\gamma$ in the interval $[t_0,t_{max})$ as 
\begin{equation*}
\Lambda=\mathcal{V}\oplus \mathcal{H}
\end{equation*}
where every nonzero Jacobi field $J\in\mathcal{H}$ is equal to $J=\tilde{f}%
\cdot E, $ 
where $E$ is a unit parallel field with $E(t_0)\in W_{t_0}$, and $\tilde{f}$
is the function from \eqref{model Jacobi} %
that satisfies $\tilde{f}\left( t_{0}\right) =\left\vert J\left(
t_{0}\right) \right\vert.$
\end{lemma}

\begin{lemma}
\label{lem:riccati comparison nnc} Let $\gamma :[t_0 ,\infty)
\longrightarrow M $ be a unit speed geodesic in a complete Riemannian $n$%
--manifold $M$ with $\newline
Ric_{k} \geq 0$. Let $\Lambda $ be a Lagrangian subspace of normal Jacobi
fields along $\gamma $ with Riccati operator $S.$ Suppose that for some $k$%
--dimensional subspace $W_{t_{0}}\perp \gamma^{\prime }(t_0) $ , 
\begin{equation}
\Trace S_{t_{0}}|_{W_{t_0}} \leq 0.  \label{small initial_nnc}
\end{equation}
With $\mathcal{V}$ and $H(t)$ as in Lemma \ref{sing soon Ric_k Lemma}, the
Jacobi equation splits orthogonally along $\gamma$ in the interval $%
[t_0,\infty)$ as 
\begin{equation*}
\Lambda=\mathcal{V}\oplus \mathcal{H}
\end{equation*}
where every nonzero Jacobi field $J\in\mathcal{H}$ is equal to $J=\tilde{f}%
\cdot E, $ 
where $E$ is a unit parallel field with $E(t_0)\in W_{t_0}$, and $\tilde{f}$
is the function from \eqref{model Jacobi} %
that satisfies $\tilde{f}\left( t_{0}\right) =\left\vert J\left(
t_{0}\right) \right\vert.$
\end{lemma}

We also need the focal radius theorem from \cite{GuijWilh}.

\begin{theorem}
\label{Intermeadiate Ricci Thm}For $k\geq 1,$ suppose that $M$ is a complete
Riemannian $n$--manifold with $\Ric_{k}\geq k$ and $N$ is any submanifold of 
$M$ with $\dim \left( N\right) \geq k.\vspace{0.1in}$

\noindent 1. Counting multiplicities, every unit speed geodesic $\gamma $
that leaves $N$ orthogonally at time $0$ has at least\ $\dim \left( N\right)
-k+1$ focal points for $N$ in $\left[ -\frac{\pi }{2},\frac{\pi }{2}\right]
. $ In particular, the focal radius of $N$ is $\leq \frac{\pi }{2}.\vspace{%
0.1in}$

\noindent 2. If $N$ has focal radius $\frac{\pi }{2},$ then it is totally
geodesic.$\vspace{0.1in}$

\noindent 3. If $N$ is closed and has focal radius $\frac{\pi }{2},$ then
the universal cover of $M$ is isometric to the sphere or a projective space
with the standard metrics, and $N$ is totally geodesic in $M.$
\end{theorem}

\section{Second Fundamental Form, Focal Radius, and Lower Curvature Bounds 
\label{focal and lower sect}}

In this section, we prove Theorems \ref{II estimate} and \ref{inter ricci
soul them}. The first is a special case of the following result.

\begin{theorem}
\label{shape est. thm}For $\kappa =-1,0,$ or $1,$ let $M$ be a complete
Riemannian $n$--manifold with $\Ric_{k}$ $\geq k\kappa ,$ and let $N$ be any
submanifold of $M$ with $\dim \left( N\right) \geq k.$ Then for any unit
normal vector $v$ to $N,$ the shape operator of $N$ for $v$ satisfies 
\begin{equation}
\begin{array}{ll}
\left\vert\, \displaystyle\sum\limits_{i=1}^{k}\left\langle \mathrm{S}%
_{v}\left( e_{i}\right) ,e_{i}\right\rangle\, \right\vert \leq k\cot \left( 
\mathrm{FocalRadius}\left( N,\gamma _{v}\right) \right) & \text{if }\kappa
=1,\vspace{0.07in} \\ 
\left\vert\, \displaystyle\sum\limits_{i=1}^{k}\left\langle \mathrm{S}%
_{v}\left( e_{i}\right) ,e_{i}\right\rangle\, \right\vert \leq \dfrac{k}{%
\mathrm{FocalRadius}\left( N,\gamma _{v}\right) } & \text{if }\kappa =0,%
\vspace{0.07in} \\ 
\left\vert \,\displaystyle\sum\limits_{i=1}^{k}\left\langle \mathrm{S}%
_{v}\left( e_{i}\right) ,e_{i}\right\rangle \,\right\vert \leq k\coth \left( 
\mathrm{FocalRadius}\left( N,\gamma _{v}\right) \right) & \text{if }\kappa
=-1,%
\end{array}
\label{trace of  S}
\end{equation}%
where $\left\{ e_{i}\right\} _{i=1}^{k}\subset T_{\gamma _{v}\left( 0\right)
}N$ is any orthonormal set and $\mathrm{FocalRadius}\left( N,\gamma
_{v}\right) $ is the focal radius of $N$ along $\gamma _{v}.$
\end{theorem}

\begin{proof}
We set%
\begin{equation*}
\mathrm{ct}_{\kappa }\left( t\right) =\left\{ 
\begin{array}{ll}
\cot \left( t\right) & \text{if }\kappa =1,\vspace{0.07in} \\ 
{1}/{t} & \text{if }\kappa =0,\vspace{0.07in} \\ 
\coth \left( t\right) & \text{if }\kappa =-1.%
\end{array}%
\right.
\end{equation*}%
Then $\mathrm{ct}_{\kappa }$ is an odd function that satisfies $%
\lim_{t\rightarrow 0^{-}}\mathrm{ct}_{\kappa }=-\infty$. 

Let $\Lambda _{N}$ be the Lagrangian family along $\gamma _{v}$ from
Equation \eqref{dfn of Lambda_N}. Let $S$ be the corresponding Riccati
operator. Observe that at $t=0,$ the restriction of $S$ to the second
summand in (\ref{dfn of Lambda_N}) coincides with the shape operator $S_{v}$
of $N$.

If the conclusion is false, there are $\left\{ J_{1},\ldots ,J_{k}\right\}
\subset \Lambda _{N}$ with $\left\{ J_{i}\left( 0\right) \right\} _{i=1}^{k}$
orthonormal and tangent to $N$ so that 
\begin{eqnarray*}
J_{i}^{\prime }\left( 0\right) ^{T} &=&S\left( J_{i}\left( 0\right) \right) 
\text{ and } \\
\left\vert \Sigma _{i=1}^{k}\left\langle S\left( J_{i}\left( 0\right)
\right) ,J_{i}\left( 0\right) \right\rangle \right\vert &\geq &k\cdot 
\mathrm{ct}_{\kappa }\left( \mathrm{FocalRadius}\left( N\right) -\alpha
\right) ,
\end{eqnarray*}%
for some $\alpha \in \left( 0,\mathrm{FocalRadius}\left( N\right) \right) .$
After possibly replacing $v$ with $-v$, we may assume that 
\begin{eqnarray}
\Sigma _{i=1}^{k}\left\langle S\left( J_{i}\right) ,J_{i}\right\rangle |_{0}
&\leq &-k\cdot \mathrm{ct}_{\kappa }\left( \mathrm{FocalRadius}\left(
N\right) -\alpha \right)  \notag \\
&=&k\cdot \mathrm{ct}_{\kappa }\left( \alpha -\mathrm{FocalRadius}\left(
N\right) \right) ,\text{ since }\mathrm{ct}_{\kappa }\text{ is an odd
function}  \notag \\
&<&0.  \label{bad trace}
\end{eqnarray}

We apply Lemma \ref{sing soon Ric_k Lemma} with $\Lambda =\Lambda _{N}$ and $%
W_{t_{0}}=\mathrm{span}\left\{ J_{i}\left( 0\right) \right\} .$ To see that
the hypotheses of Lemma \ref{sing soon Ric_k Lemma} are satisfied, we note
that:

\begin{itemize}
\item Inequality (\ref{bad trace}) gives us that Inequality (\ref{small
initial}) holds with 
\begin{equation*}
\tilde{\lambda}_{\kappa }\left( t\right) =\mathrm{ct}_{\kappa }\left( \alpha
-\mathrm{FocalRadius}\left( N\right) +t\right) \text{ and }t_{0}=0.
\end{equation*}

\item Since $\Lambda _{N}$ is nonsingular on $\left( 0,\mathrm{FocalRadius}%
\left( N\right) \right) ,$ its subspace $\mathcal{V}$ has full index on the
interval $\left( 0,\mathrm{FocalRadius}\left( N\right) \right) .$
\end{itemize}

Thus it follows from Lemma \ref{sing soon Ric_k Lemma} that for all $%
t_{1}\in \left( 0,\mathrm{FocalRadius}\left( N\right) \right) $ there is a $%
k $--dimensional subspace 
$H(t_1)\subset \gamma^{\prime }(t_1)$  so that 
\begin{equation*}
\Tr S|_{H(t_1)} \leq k\cdot \mathrm{ct}_{\kappa }\left( \alpha -\mathrm{%
FocalRadius}\left( N\right) +t_{1}\right) .
\end{equation*}%
Since $\alpha -\mathrm{FocalRadius}\left( N\right) <0$ and $%
\lim_{t\rightarrow 0^{-}}\mathrm{ct}_{\kappa }=-\infty ,$ $\Lambda _{N}$ has
a singularity by time $\mathrm{FocalRadius}\left( N\right) -\alpha .$ This
is a contradiction because $\Lambda _{N}$ is nonsingular on the interval $%
\left( 0,\mathrm{FocalRadius}\left( N\right) \right) $.
\end{proof}

\begin{proof}[Proof of Theorem \protect\ref{inter ricci soul them}]
Let $M$ be a complete Riemannian $n$--manifold with $\Ric_{k}$ $\geq 0,$ and
let $N$ be any closed submanifold of $M$ with $\dim \left( N\right) \geq k$
and infinite focal radius. Let $v$ be any unit normal vector to $N.$ As in
Equation \eqref{dfn of Lambda_N} we let%
\begin{equation*}
\Lambda _{N}\equiv \left\{ J|J\left( 0\right) =0\text{, }J^{\prime }\left(
0\right) \in \nu _{\gamma _{v}\left( 0\right) } N \right\} \oplus \left\{
J|J\left( 0\right) \in T_{\gamma _{v}\left( 0\right) }N\text{ and }J^{\prime
}\left( 0\right) =\mathrm{S}_{v}J\left( 0\right) \right\} .
\end{equation*}%
We set%
\begin{eqnarray*}
\mathcal{V} &\equiv &\left\{ J|J\left( 0\right) =0\text{, }J^{\prime }\left(
0\right) \in \nu _{\gamma _{v}\left( 0\right) } N \right\} \text{, and } \\
W &\equiv &\left\{ J|J\left( 0\right) \in T_{\gamma _{v}\left( 0\right) }N%
\text{ and }J^{\prime }\left( 0\right) =\mathrm{S}_{v}J\left( 0\right)
\right\} .
\end{eqnarray*}

Since $N$ has no focal points, $\Lambda _{N}$ has no singularities on $%
\mathbb{R}\setminus \left\{ 0\right\} $. Thus for all $J\in \Lambda
_{N}\setminus \left\{ 0\right\} $ and all $t\in \mathbb{R}\setminus \left\{
0\right\} ,$ $J\left( t\right) \neq 0.$ By replacing $v$ with $-v$, if
necessary, we may assume that 
\begin{equation}
\Tr\left( \mathrm{S}|_{W}\left( 0\right) \right) \leq 0.  \label{inti nonpos}
\end{equation}

By Lemma \ref{lem:riccati comparison nnc}, it follows that $t\longmapsto
\Lambda \left( t\right) $ splits orthogonally into the parallel
distributions 
\begin{equation*}
\Lambda \left( t\right) \equiv W\left( t\right) \oplus \mathcal{V}\left(
t\right) ,
\end{equation*}%
and every field in $W$ is parallel. Since we started with an arbitrary
normal vector$,$ $N$ is totally geodesic, and Parts 1 and 4 are proven. Part
2 is a consequence of the Hopf-Rinow Theorem (see Part (e) of Theorem 2.8 on
page 147 of \cite{doCarm}).

Part 3 follows by observing that $\exp _{N}^{\perp }:\left( \nu \left(
N\right) ,\left( \exp _{N}^{\perp }\right) ^{\ast }\left( g\right) \right)
\longrightarrow \left( M,g\right) $ is a local isometry, so as in the proof
of Cartan-Hadamard, $\exp _{N}^{\perp }$ is a cover (see Lemma 3.3 on page
150 of \cite{doCarm} or Lemma 5.6.4 of \cite{Pet}). Part 5 follows from the
fact that every field in $W$ is parallel.

To prove Part 6, let $\mathrm{II}$ be the second fundamental form of $%
\mathrm{image}\left( \Phi \right) .$ Since 
\begin{equation*}
\gamma _{V\left( s\right) }:t\longmapsto \Phi \left( t,s\right) =\exp
_{c\left( s\right) }^{\perp }\left( tV\left( s\right) \right)
\end{equation*}%
is a geodesic, $\mathrm{II}\left( \frac{\partial \Phi }{\partial t},\frac{%
\partial \Phi }{\partial t}\right) =0,$ and since $\frac{\partial \Phi }{%
\partial s}$ is parallel along $\gamma _{V\left( s\right) },$ $\mathrm{II}%
\left( \frac{\partial \Phi }{\partial t},\frac{\partial \Phi }{\partial s}%
\right) =0.$ To determine $\mathrm{II}\left( \frac{\partial \Phi }{\partial s%
},\frac{\partial \Phi }{\partial s}\right) ,$ observe that the lift, $\left(
\exp _{N}^{\perp }\right) ^{\ast }\left( \frac{\partial \Phi }{\partial s}%
\right) ,$ of $\frac{\partial \Phi }{\partial s}$ via $\exp _{N}^{\perp },$
is a basic horizontal, geodesic field for the Riemannian submersion%
\begin{equation*}
\pi :(\nu \left( N\right) ,\left( \exp _{N}^{\perp }\right) ^{\ast }\left(
g\right) )\longrightarrow N.
\end{equation*}%
Thus $\mathrm{II}\left( \frac{\partial \Phi }{\partial s},\frac{\partial
\Phi }{\partial s}\right) =\nabla _{\frac{\partial \Phi }{\partial s}}\frac{%
\partial \Phi }{\partial s}\equiv 0,$ and the image of $\Phi $ is totally
geodesic. Since $\frac{\partial \Phi }{\partial s}$ is a parallel Jacobi
field along $\gamma _{V\left( s\right) },$ the image of $\Phi $ is flat.

To prove Part 7, consider a $V\in \mathcal{V}$ along with orthonormal
parallel fields $J_{1},\ldots ,J_{k-1}$ in $W.$ Since $\mathrm{sec}\left(
\gamma _{v}^{\prime },J_{i}\right) \equiv 0$ and $\mathrm{Ric}_{k}\geq 0,$ $%
\mathrm{sec}\left( \gamma _{v}^{\prime },V\right) \geq 0.$ Since $\Lambda
\left( t\right) \equiv W\left( t\right) \oplus \mathcal{V}\left( t\right) $
is a parallel, orthogonal splitting, all curvatures of the form $\mathrm{sec}%
\left( \gamma _{v}^{\prime },\cdot \right) $ are nonnegative.

Since $\mathrm{Ric}_{k}M\geq 0,$ it follows from the Gauss Equation that to
prove Part 8, it suffices to show that 
\begin{equation*}
\left\langle S\left( J\right) ,J\right\rangle \geq 0,
\end{equation*}%
for all $J\in \Lambda _{N}$ and all $t\geq 0.$ Since $\Lambda _{N}\left(
t\right) \equiv W\left( t\right) \oplus \mathcal{V}\left( t\right) $ is a
parallel, orthogonal splitting and $\left\langle S\left( J\right)
,J\right\rangle \equiv 0$ for all $J\in W,$ it suffices to show that 
\begin{equation*}
\left\langle S\left( J\right) ,J\right\rangle \geq 0
\end{equation*}%
for all $J\in \mathcal{V}$ and all $t\geq 0.$ If not, then for some $t_{0}>0$
and some $J\in \mathcal{V},$ $\left\langle S\left( J\right) ,J\right\rangle
<0.$ Set 
\begin{equation*}
U\equiv \left\{ J,L_{1},\ldots ,L_{k-1}\right\} ,
\end{equation*}%
where $L_{1},\ldots ,L_{k-1}$ are $\left( k-1\right) $--linearly independent
fields of $W.$ It follows that for some $c>t_{0},$ 
\begin{equation*}
Tr\left( S|_{U}\right) \left( t_{0}\right) <\frac{1}{t_{0}-c}<0,
\end{equation*}%
and hence from Lemma \ref{lem:riccati comparison nnc} that $\Lambda _{N}$
has a singularity, which is contrary to our hypothesis that $N$ has infinite
focal radius.
\end{proof}

In the case that $M$ is not simply connected, we have the following
structure result.

\begin{corollary}
\label{inf focal nonneg cor}Let $N$ be a closed submanifold in a compact
nonnegatively curved manifold $M.$ If $N$ has infinite focal radius, then
the universal cover, $\tilde{M}$ splits isometrically as 
\begin{equation*}
\tilde{M}=\tilde{N}_{0}\times \mathbb{R}^{m}
\end{equation*}%
where $\tilde{N}_{0}$ is compact and simply connected, and the universal
cover $\tilde{N}$ of $N$ is isometrically embedded in $\tilde{M}$ as 
\begin{equation*}
\tilde{N}=\tilde{N}_{0}\times \mathbb{R}^{l},
\end{equation*}%
where $\mathbb{R}^{l}$ is an affine subspace of $\mathbb{R}^{m}.$
\end{corollary}

\begin{proof}
Let $\pi :\tilde{M}\longrightarrow M$ by the universal cover. By Theorem 9.1
in \cite{CheegGrom}, $\tilde{M}$ splits isometrically as 
\begin{equation*}
\tilde{M}=M_{0}\times \mathbb{R}^{m}
\end{equation*}%
where $M_{0}$ is compact and simply connected. By Part 3 of Theorem \ref%
{inter ricci soul them}, $\pi :\tilde{M}\longrightarrow M$ factors through $%
\exp _{N}^{\perp }:\nu \left( N\right) \longrightarrow M.$ That is we have a
Riemannian cover $p:\tilde{M}\longrightarrow \nu \left( N\right) $ so that $%
\pi =\exp _{N}^{\perp }\circ p.$ Since every normal vector to the zero
section, $N_{0},$ in $\nu \left( N\right) $ exponentiates to a ray, every
normal vector to $\pi ^{-1}\left( N\right) $ exponentiates to a ray. Since $%
\tilde{M}$ is the metric product, $M_{0}\times \mathbb{R}^{m},$ every normal
vector to $\pi ^{-1}\left( N\right) $ is tangent to an $\mathbb{R}^{m}$%
--factor. Thus every tangent space to $\pi ^{-1}\left( N\right) $ has the
form $TM_{0}\times \mathbb{R}^{l},$ where $\mathbb{R}^{l}$ is an affine
subspace of $\mathbb{R}^{m}.$ Since $\pi ^{-1}\left( N\right) $ is totally
geodesic and without boundary it follows that $\pi ^{-1}\left( N\right) $ is 
$M_{0}\times \mathbb{R}^{l}$ where $\mathbb{R}^{l}$ is an affine subspace of 
$\mathbb{R}^{m}.$
\end{proof}

\subsection{What can be done with just classical Riccati comparison?\label%
{subsect: classic comp}}

Although weak versions of all of our results can be obtained using just
classical Riccati comparison, to the best of our knowledge no theorem
discussed here can be proven with out the Transverse Jacobi Equation. As a
concrete example, we point out that classical comparison yields the
following weak form of Theorem \ref{II estimate}.

\bigskip

\noindent \textbf{Weak Form of Theorem \ref{II estimate}: }\emph{For }$\kappa =-1,0,$\emph{\
or }$1,$\emph{\ let }$M$\emph{\ be a complete Riemannian }$n$\emph{%
--manifold with sectional curvature }$\geq \kappa ,$\emph{\ and let }$N$%
\emph{\ be any hypersurface of }$M.$\emph{\ Then at every point of }$N$\emph{%
\ there is a \textbf{single} vector }$v$\emph{\ so that }%
\begin{equation*}
\begin{array}{ll}
\,\mathrm{II}_{N}\left( v,v\right) \geq -\dfrac{\left\vert v\right\vert ^{2}%
}{\cot \left( \mathrm{FocalRadius}\left( N\right) \right) } & \text{if }%
\kappa =1 \\ 
\mathrm{II}_{N}\left( v,v\right) \geq -\dfrac{\left\vert v\right\vert ^{2}}{%
\mathrm{FocalRadius}\left( N\right) } & \text{if }\kappa =0 \\ 
\mathrm{II}_{N}\left( v,v\right) \geq -\dfrac{\left\vert v\right\vert ^{2}}{%
\coth \left( \mathrm{FocalRadius}\left( N\right) \right) } & \text{if }%
\kappa =-1.%
\end{array}%
\end{equation*}

\bigskip

To clarify how classical comparison fails to yield Theorem \ref{II estimate}, we note that
the sectional curvature version of Lemma $\ref{sing soon Ric_k Lemma}$
implies that if Inequality ($\ref{small future Inequal})$ fails for all $1$%
--dimensional subspaces $H\left( t\right) \subset T_{\gamma \left( t\right)
}M ^{\perp }$ then Inequality \eqref{small initial} fails for all $1$%
--dimensional subspaces $W_{t_{0}}\subset T_{\gamma \left( t_{0}\right)
}M^{\perp }$. In contrast, the classical theorem of \cite{EschHein} only
gives that Inequality ($\ref{small initial})$ fails for some $%
W_{t_{0}}\subset T_{\gamma \left( t_{0}\right) }M ^{\perp }$. Examples 2.37
and 2.38 in \cite{GuijWilh} show that there is no classical analog to Lemma $%
\ref{sing soon Ric_k Lemma},$ (also see the commentary after Lemma E in \cite%
{GuijWilh}.)

\section{Submanifold Restrictions\label{submanf finitieness secti}}

The main step in the proof of Theorem \ref{finiteness thm} is to show that
the intrinsic metrics on all of the submanifolds satisfy the hypothesis of
Cheeger's Finiteness Theorem, \cite{Cheeg2}.

\begin{lemma}
\label{bnd geom lemma}Let $M$ be a compact Riemannian manifold. Given $D,r>0$
let $\mathcal{S}$ be the class of closed Riemannian manifolds that can be
isometrically embedded into $M$ with focal radius $\geq r$ and intrinsic
diameter $\leq D$. Then there are positive numbers $K,v>0$ so that for every 
$S\in $ $\mathcal{S}$,%
\begin{equation*}
\left\vert \mathrm{sec}_{S}\right\vert \leq K\text{ and \textrm{vol}}\left(
S\right) >v.
\end{equation*}
\end{lemma}

\begin{proof}
The compactness of $M$ gives us ambient upper and lower curvature bounds.
Combined with Theorem \ref{II estimate}, we get the existence of $K.$

It remains to derive a uniform lower volume bound for the $S\in \mathcal{S}.$
To do this we use the first display formula on Page 1 of \cite{HeKar}:%
\begin{equation*}
\mathrm{vol}\left( M\right) \leq \mathrm{vol}\left( N\right) \cdot f_{\delta
}\left( \mathrm{diam}\left( M\right) ,\Lambda \right) .
\end{equation*}%
Here $N$ is a compact, embedded submanifold of $M,$ $\delta $ is a lower
curvature bound for $M$, $\Lambda $ is an upper bound for the mean curvature
of $N,$ and the function $f_{\delta }$ is given explicitly on Page 453 of 
\cite{HeKar}. Theorem \ref{II estimate} gives us an upper bound for $\Lambda 
$ and hence a $C>0$ so that 
\begin{equation*}
f_{\delta }\left( \mathrm{diam}\left( M\right) ,\Lambda \right) \leq C.
\end{equation*}%
Setting $v=\frac{\mathrm{vol}\left( M\right) }{C}$ completes the proof.
\end{proof}

Recall that the Cheeger-Gromov compactness theorem states

\begin{theorem}[{see \protect\cite[Theorem 3.6]{Cheeg1}, \protect\cite[%
Theorem 11.3.6]{Pet}}]
Given $0<\beta <\alpha <1,$ $k,K\in \mathbb{R},$ $v,D>0,$ and $n\in \mathbb{N%
},$ let $\left\{ M_{i}\right\} _{i=1}^{\infty }$ be a sequence of closed
Riemannian manifolds with 
\begin{equation*}
k\leq \mathrm{sec}M_{i}\leq K,\text{ }\mathrm{vol}\left( M_{i}\right) \geq v,%
\text{ and }\mathrm{Diam}\left( M_{i}\right) \leq D.
\end{equation*}%
Then there is a $C^{1,\alpha }$--Riemannian manifold $M_{\infty }$ and a
subsequence of $\left\{ M_{i}\right\} _{i=1}^{\infty }$ that converges to $%
M_{\infty }$ in the $C^{1,\beta }$ topology.
\end{theorem}

\begin{proof}[Proof of Theorem \protect\ref{finiteness thm}]
It follows from Lemma \ref{bnd geom lemma} that the class $\mathcal{S}$
satisfies the hypotheses of Cheeger's finiteness theorem. So any sequence $%
\left\{ S_{i}\right\} \subset \mathcal{S}$ has a subsequence (also called $%
\left\{ S_{i}\right\} $) that converges in the $C^{1,\beta }$--topology to
an abstract $C^{1,\alpha }$ Riemannian manifold $\left( S_{\infty
},g_{\infty }\right) .$ Let $\varphi _{i}:S_{\infty }\longrightarrow S_{i}$
be diffeomorphisms so that $\varphi _{i}^{\ast }\left( g_{i}\right) \overset{%
C^{1,\beta }}{\longrightarrow }$ $g_{\infty }.$ Let $f_{i}:S_{i}$ $%
\longrightarrow M$ be the sequence of inclusions of $S_{i}$ into $M.$
Composing gives a sequence $f_{i}\circ \varphi _{i}:S_{\infty
}\longrightarrow M,$ that is uniformly bounded in the $C^{1,\beta }$%
--topology. From Arzela-Ascoli it follows that $\left\{ f_{i}\circ \varphi
_{i}\right\} _{i}$ subconverges in the $C^{1,\beta }$--topology to an
isometric embedding $f_{\infty }:S_{\infty }\longrightarrow M$.
\end{proof}

\begin{example}[Theorem \protect\ref{finiteness thm} is optimal]
\label{focal ex} The isometric embedding theorem of J. Nash says that for
given $k$, there is some $n=n(k)$ such that any $k$-dimensional Riemannian
manifold embeds isometrically in $\mathbb{R}^{n}$. Consider then any compact
Riemannian manifold, and rescale its metric so that its diameter is bounded
above by 1. If needed, rescale the metric in $\mathbb{R}^{n}$ so that the
image of an isometric embedding $f:M\rightarrow \mathbb{R}^{n}$ is contained
in the interior of some fundamental domain for the covering space $\pi :%
\mathbb{R}^{n}\rightarrow \mathbb{T}^{n}$. Taking the composition $\pi \circ
f$, we get an isometric embedding of $M$ into $\mathbb{T}^{n}$ with
intrinsic diameter bounded above; thus Theorem \ref{finiteness thm} is
optimal in the sense that its conclusion is false if the hypothesis about
the lower bound on the focal radii is removed. 

To see that the hypothesis about the intrinsic diameters can not be removed,
let $\lambda \mathbb{S}^{1}$ be the circle of radius $\lambda $. For each $k$%
-manifold $(M,g)$, choose a rational number $\lambda $ so that the image of
the isometric embedding 
\begin{equation*}
j:\lambda M\hookrightarrow \lambda \mathbb{T}^{n}
\end{equation*}%
has focal radius greater or equal than 1.

Next use that, for the given $\lambda $, there is an isometric embedding 
\begin{equation*}
\iota :\lambda \mathbb{S}^{1}\hookrightarrow \mathbb{T}^{2},
\end{equation*}%
and let 
\begin{equation*}
I:\lambda \mathbb{T}^{n}\rightarrow \mathbb{T}^{2n}
\end{equation*}%
be the product embedding. The images of the composition $I\circ j:\left(
M,g\right) \hookrightarrow \mathbb{T}^{2n}$ all have focal radius $\geq 1$.
Thus Theorem \ref{finiteness thm} is false if the hypothesis about the upper
bound on the diameter is removed.
\end{example}

%
%

\subsection{Other Finiteness Statements}

Various other finiteness theorems for submanifolds follow by combining the
proof of Theorem \ref{finiteness thm} with existing results. For example,
using the main theorem of \cite{GrovPetWu} we have

\begin{theorem}
Given $k\in \mathbb{R},$ $v,D>0,$ $n\in \mathbb{N},$ and $r>0,$ let $%
\mathcal{M}\left( k,v,n\right) $ be the class of closed Riemannian $n$%
-manifolds with sectional curvature $\geq k$ , volume $\geq v,$ and diameter 
$\leq D,$ and let $\mathcal{S}$ be the class of closed Riemannian manifolds
that can be isometrically embedded into an element of $\mathcal{M}\left(
k,v,n\right) $ so that the image has focal radius $\geq r$ and intrinsic
diameter $\leq D.$ Then $\mathcal{S}$ contains only finitely many
homeomorphism types.
\end{theorem}

\section{Submetries and Conjugate Points\label{subm and conj sect}}

In this section we prove Theorem \ref{submetry cor}. We start, in subsection %
\ref{submetries and Holn sect} with a establishing some basic facts about
holonomy for manifold submetries. We then prove Theorem \ref{submetry cor}
in subsection \ref{Subm and conj pts}.

\subsection{Submetries and Holonomy\label{submetries and Holn sect}}

Throughout this section, we assume $M$ is an Alexandrov space with curvature
bounded from below, $\pi :M\longrightarrow X$ is a submetry, and $\gamma :%
\left[ 0,b\right] \longrightarrow X$ is a geodesic.

The proof of Lemma 2.1 in \cite{BereGuij} gives us the following.

\begin{proposition}
\noindent 1. Given any $y\in \pi ^{-1}\left( \gamma \left( 0\right) \right)
, $ there is a lift of $\gamma $ starting at $y.$

\noindent 2. If for some $\varepsilon >0,$ $\gamma $ extends as a geodesic
to $\left[ -\varepsilon ,b\right] ,$ then the lift in Part 1 is unique.
\end{proposition}

Part 2 allows us to define holonomy maps between the fibers of $\pi $ over
the \emph{interior} of $\gamma $ as follows.

\begin{definition}
Given any $s,t\in \left( 0,b\right) ,$ we define the holonomy maps 
\begin{equation*}
H_{s,t}:\pi ^{-1}\left( \gamma \left( s\right) \right) \longrightarrow \pi
^{-1}\left( \gamma \left( t\right) \right)
\end{equation*}%
by 
\begin{equation*}
H_{s,t}\left( x\right) =\tilde{\gamma}_{x}\left( t\right) ,
\end{equation*}%
where $\tilde{\gamma}_{x}$ is the unique lift of $\gamma $ so that $\tilde{%
\gamma}_{x}\left( s\right) =x.$
\end{definition}

\begin{proposition}
\label{Hol diffeos}If $M$ is Riemannian and $\pi $ is a manifold submetry,
then for all $s,t\in \left( 0,b\right) ,$ $H_{s,t}$ is a $C^{\infty }$
diffeomorphism.
\end{proposition}

\begin{proof}
Choose $\varepsilon _{0}>0$ so that $\left[ s-\varepsilon _{0},t+\varepsilon
_{0}\right] \subset \left( 0,b\right) .$ By compactness we cover $\left[ s,t%
\right] $ by finite number of open intervals of the form 
\begin{equation*}
\left( s_{i}-\iota _{i},s_{i}+\iota _{i}\right) ,
\end{equation*}%
were $\iota _{i}$ is one-fourth of the injectivity radius of $\pi
^{-1}\left( \gamma \left( s_{i}\right) \right) $, and 
\begin{equation*}
s-\varepsilon _{0}=s_{0}<s_{1}<\cdots <s_{m}=t+\varepsilon _{0}.
\end{equation*}

Let $F_{t}^{i}$ be the flow of $\mathop{\rm grad}\mathrm{dist}_{\pi
^{-1}\left( \gamma \left( s_{i}\right) \right) }.$ Then for $r_{1},r_{2}\in
\left( s_{i},s_{i+1}+\iota _{i+1}\right) ,$ $H_{r_{1},r_{2}}$ is the
restriction of $F_{r_{2}-r_{1}}^{i}$ to $\pi ^{-1}\left( \gamma \left(
r_{1}\right) \right) $ and hence is a diffeomorphism onto its image $\pi
^{-1}\left( \gamma \left( r_{2}\right) \right) $. Since $H_{s,t}$ is the
composition of a finite number of the diffeomorphisms $H_{r_{1},r_{2}},$ it
follows that $H_{s,t}$ is a diffeomorphism.
\end{proof}

\begin{remark}
\label{holo rem}For $\gamma $ and $\tilde{\gamma}_{x}$ as above, we define
the holonomy fields along $\tilde{\gamma}_{x}$ to be the Jacobi fields that
correspond to variations by lifts of $\gamma .$ If the Lagrangian subspace $%
\Lambda _{\pi ^{-1}\left( \gamma \left( s\right) \right) }$ has no
singularities on $\left( s,t\right) ,$ that is, if the evaluation map $%
\mathcal{E}_{u}:\Lambda _{\pi ^{-1}\left( \gamma \left( s\right) \right)
}\longrightarrow T_{\gamma \left( u\right) }M$ is one-to-one for all $u\in
\left( s,t\right) $, it follows that a field $J\in \Lambda _{\pi ^{-1}\left(
\gamma \left( s\right) \right) }$ is holonomy if $J\left( u\right) \in T\pi
^{-1}\left( \gamma \left( u\right) \right) $ for some $u\in \left(
s,t\right) .$
\end{remark}

\subsection{Submetries and Variational Conjugate Points\label{Subm and conj
pts}}

The following is the precise sense in which the term \textquotedblleft
conjugate point\textquotedblright\ is used in Theorem \ref{submetry cor}.

\begin{definition}
\label{variat conj dfn}(Variational Conjugate Point) Let $\gamma :\left[ 0,b%
\right] \longrightarrow X$ be a unit speed geodesic in a complete, locally
compact length space $X.$ We say that $\gamma \left( b\right) $ is
variationally conjugate to $\gamma \left( 0\right) $ along $\gamma $ if and
only if for some $\varepsilon >0,$ there is a continuous map $V:\left[ 0,b%
\right] \times \left( -\varepsilon ,\varepsilon \right) \longrightarrow X$
with the following properties.

\noindent 1. For all $t\in \left( 0,b\right) ,$ 
\begin{equation*}
\gamma \left( t\right) =V\left( t,0\right) .
\end{equation*}%
\noindent 2. There is a $C>0$ and a $t_{0}\in \left( 0,b\right) $ so that
for all sufficiently small $s\neq 0$,%
\begin{equation*}
\mathrm{dist}\left( \gamma \left( t_{0}\right) ,V\left( t_{0},s\right)
\right) \geq Cs.
\end{equation*}

\noindent 3. For each $s\in \left( -\varepsilon ,\varepsilon \right) ,$ 
\begin{equation*}
t\mapsto V\left( t,s\right)
\end{equation*}%
is a unit speed geodesic on $\left[ 0,b\right] .$

\noindent 4. At the end points, 
\begin{eqnarray*}
\mathrm{dist}\left( V\left( 0,0\right) ,V\left( 0,s\right) \right) &\leq
&o\left( s\right) \text{ and} \\
\mathrm{dist}\left( V\left( b,0\right) ,V\left( b,s\right) \right) &\leq
&o\left( s\right) .
\end{eqnarray*}
\end{definition}

In the Riemannian case, this coincides with the usual definition of
conjugacy, so it is not surprising that geodesics in Alexandrov spaces stop
minimizing distance after variational conjugate points.

\begin{proposition}
\label{conj not min prop}If $X$ is an Alexandrov space with curvature
bounded from below and $\gamma \left( b\right) $ is variationally conjugate
to $\gamma \left( 0\right) $ along $\gamma ,$ then for all $\varepsilon >0,$
either $\gamma $ does not extend to $\left[ 0,b+\varepsilon \right] $ or $%
\gamma |_{\left[ 0,b+\varepsilon \right] }$ is not minimal.
\end{proposition}

\begin{proof}
Suppose that $\gamma |_{\left[ 0,b+\eta \right] }$ is minimal and that $\eta 
$ is small enough so that $t_{0}\in \left( \eta ,b-\eta \right) .$ Since the
comparison angle $\tilde{\sphericalangle}\left( \gamma \left( 0\right)
,\gamma \left( b\right) ,\gamma \left( b+\eta \right) \right) $ is $\pi ,$
it follows that 
\begin{eqnarray*}
\sphericalangle \left( \gamma \left( 0\right) ,V\left( b,s\right) ,\gamma
\left( b+\eta \right) \right) &\geq &\tilde{\sphericalangle}\left( \gamma
\left( 0\right) ,V\left( b,s\right) ,\gamma \left( b+\eta \right) \right)
>\pi -o\left( \frac{s}{\eta }\right) \text{,} \\
\sphericalangle \left( \gamma \left( t_{0}\right) ,V\left( b,s\right)
,\gamma \left( b+\eta \right) \right) &\geq &\tilde{\sphericalangle}\left(
\gamma \left( t_{0}\right) ,V\left( b,s\right) ,\gamma \left( b+\eta \right)
\right) >\pi -o\left( \frac{s}{\eta }\right) .
\end{eqnarray*}

The previous two inequalities, together with a hinge comparison argument in
the \emph{space of directions} of $X$ at $V\left( b,s\right) ,$ gives 
\begin{equation*}
\sphericalangle \left( \gamma \left( 0\right) ,V\left( b,s\right) ,\gamma
\left( t_{0}\right) \right) \leq o\left( \frac{s}{\eta }\right) .
\end{equation*}%
So by hinge comparison in $X,$%
\begin{equation*}
\mathrm{dist}\left( V\left( t_{0},s\right) ,\gamma \left( t_{0}\right)
\right) \leq o\left( \frac{s}{\eta }\right) ,
\end{equation*}%
but this is contrary to Part 2 of the definition of variational conjugacy.
\end{proof}

\begin{lemma}
\label{focal implies conj}Let $\pi :M\longrightarrow X$ be a manifold
submetry. Let $\gamma :\left[ 0,b\right] \longrightarrow X$ be a geodesic,
and let $\tilde{\gamma}$ be a horizontal lift of $\gamma $ that has its
first focal point for $\pi ^{-1}\left( \gamma \left( 0\right) \right) $ at $%
b_{0}\in \left( 0,b\right) .$ Then $\gamma $ has a variational conjugate
point at $b_{0}.$
\end{lemma}

\begin{proof}
Since $\tilde{\gamma}$ has its first focal point for $\pi ^{-1}\left( \gamma
\left( 0\right) \right) $ at $b_{0},$ there is a variation 
\begin{equation*}
\tilde{V}:\left[ 0,b_{0}\right] \times \left( \varepsilon ,\varepsilon
\right) \longrightarrow M
\end{equation*}%
of $\tilde{\gamma}$ by geodesics that leave $\pi ^{-1}\left( x\right) $
orthogonally at time $0$ with%
\begin{equation}
\tilde{V}\left( 0,s\right) \in \pi ^{-1}\left( x\right) ,\text{ }\left. 
\frac{\partial }{\partial s}\tilde{V}\right\vert _{\left( b_{0},0\right) }=0%
\text{ and }\left. \frac{\partial }{\partial s}\tilde{V}\right\vert _{\left(
t,0\right) }\neq 0  \label{s deriv of V tilde}
\end{equation}%
for all $t\in \left( 0,b_{0}\right) .$ If $\frac{\partial }{\partial s}%
\tilde{V}\left( t,0\right) $ is vertical for all $t\in \left( 0,b_{0}\right)
,$ then by Remark \ref{holo rem}, $\tilde{V}$ is a holonomy field. In this
event, since $b_{0}\in \left( 0,b\right) $, it follows from Proposition \ref%
{Hol diffeos} that 
\begin{equation*}
\left. \frac{\partial }{\partial s}\tilde{V}\right\vert _{\left(
b_{0},0\right) }\neq 0,
\end{equation*}%
which is contrary to the second equation in \eqref{s deriv of V tilde}. So
for some $t_{0}\in \left( 0,b_{0}\right) $,%
\begin{equation}
\left. \frac{\partial }{\partial s}\tilde{V}\right\vert _{\left(
t_{0},0\right) }\text{ is not vertical}.  \label{not vert st}
\end{equation}

Projecting $\tilde{V}$ under $\pi $ produces a variation $V$ of $\gamma $ in 
$X$ by geodesics. It follows from \eqref{not vert st} that for all
sufficiently small $s\neq 0,$ there is a $C>0$ so that 
\begin{equation*}
\mathrm{dist}\left( \gamma \left( t_{0}\right) ,V\left( t_{0},s\right)
\right) \geq Cs.
\end{equation*}

Since 
\begin{eqnarray*}
\tilde{V}\left( 0,s\right) &\in &\pi ^{-1}\left( x\right) , \\
\mathrm{dist}\left( \tilde{V}\left( b_{0},0\right) ,\tilde{V}\left(
b_{0},s\right) \right) &\leq &o\left( s\right) ,
\end{eqnarray*}%
and $\pi $ is distance nonincreasing$,$ 
\begin{eqnarray*}
V\left( 0,0\right) &=&x\text{ and} \\
\mathrm{dist}\left( V\left( b_{0},0\right) ,V\left( b_{0},s\right) \right)
&\leq &o\left( s\right) .
\end{eqnarray*}%
Thus $\gamma $ has a variational conjugate point at time $b_{0}.$
\end{proof}

\begin{proof}[Proof of Theorem \protect\ref{submetry cor}]
Suppose that 
\begin{equation*}
\pi :M\longrightarrow X
\end{equation*}%
is a manifold submetry of a complete Riemannian $n$--manifold with $\Ric%
_{k}\geq k$ and that for some $x\in X,$ $\dim \pi ^{-1}\left( x\right) \geq
k.$ Let $\gamma $ be a geodesic of $X$ emanating from $x.$ Suppose that $%
\gamma $ extends to an interval $I$ that properly contains $\left[ -\frac{%
\pi }{2},\frac{\pi }{2}\right] .$ Then $\gamma $ is defined on $\left[ -%
\frac{\pi }{2},\frac{\pi }{2}\right] $ and either extends past $\frac{\pi }{2%
}$ or extends past $-\frac{\pi }{2}.$ Without loss of generality, assume
that $\gamma $ extends past $\frac{\pi }{2}.$ By Part 1 of Theorem \ref%
{Intermeadiate Ricci Thm}, every horizontal lift of $\gamma $ has its first
focal point for $\pi ^{-1}\left( x\right) $ at some $t_{0}\in \left[ -\frac{%
\pi }{2},\frac{\pi }{2}\right] .$ If $\tilde{\gamma}$ is such a lift and $%
t_{0}\in \left( -\frac{\pi }{2},\frac{\pi }{2}\right] ,$ then by Lemma \ref%
{focal implies conj}, $\gamma \left( t_{0}\right) $ is variationally
conjugate to $\gamma \left( 0\right) .$ If $t_{0}=-\frac{\pi }{2},$ then for
convenience, we reorient $\gamma $ so that it extends past $-\frac{\pi }{2}$
and has its first focal point at $\frac{\pi }{2}.$ Applying Part 2 of Lemma %
\ref{sing soon Ric_k Lemma} with $\kappa =1,$ $t_{0}=0,$ $t_{\max }=\frac{%
\pi }{2},$ $\tilde{\lambda}=\cot \left( t+\frac{\pi }{2}\right) $, and 
\begin{equation*}
W_{t_{0}}=\left\{ J|J\left( 0\right) \in T_{\tilde{\gamma}\left( 0\right)
}\pi ^{-1}\left( x\right) \text{ and }J^{\prime }\left( 0\right) =\mathrm{S}%
_{\tilde{\gamma}^{\prime }\left( 0\right) }J\left( 0\right) \right\} ,
\end{equation*}%
we see that $W_{t_{0}}$ is spanned by Jacobi fields of the form $\sin \left(
t+\frac{\pi }{2}\right) E,$ where $E$ is a parallel field. In particular, $%
\mathrm{S}_{\tilde{\gamma}^{\prime }\left( 0\right) }\equiv 0.$ So we can
apply Part 1 of Lemma \ref{sing soon Ric_k Lemma} and conclude that $\tilde{%
\gamma}$ also has a focal point at $s_{0}\in \left[ -\frac{\pi }{2},0\right)
.$ Since $\gamma $ extends past $-\frac{\pi }{2},$ by Theorem \ref{focal
implies conj}, $\gamma \left( s_{0}\right) $ is variationally conjugate to $%
\gamma \left( 0\right) $.

If all geodesics emanating from $x\in X$ extend to $\left[ -\frac{\pi }{2},%
\frac{\pi }{2}\right] $ and are free of variational conjugate points on $%
\left( -\frac{\pi }{2},\frac{\pi }{2}\right) $, then by Lemma \ref{focal
implies conj}, $\pi ^{-1}\left( x\right) $ has focal radius $\geq \frac{\pi 
}{2}.$ So if $\dim \pi ^{-1}\left( x\right) \geq k,$ then by Part 3 of
Theorem \ref{Intermeadiate Ricci Thm}, the universal cover of $M$ is
isometric to the sphere or a to projective space with the standard metrics,
and $\pi ^{-1}\left( x\right) $ is totally geodesic in $M.$

To prove Part 3 of Theorem \ref{submetry cor}, suppose that $p,q\in X$ are
at maximal distance $>\frac{\pi }{2},$ and $\dim \pi ^{-1}\left( p\right)
\geq k.$ Since $M$ is a compact Riemannian manifold and 
\begin{equation*}
\pi :M\longrightarrow X
\end{equation*}%
is a submetry, $X$ is an Alexandrov space with some lower curvature bound.
Since $p$ and $q$ are at maximal distance, $\pi ^{-1}\left( p\right) $ and $%
\pi ^{-1}\left( q\right) $ are at maximal distance. It follows that for any $%
\tilde{p}\in \pi ^{-1}\left( p\right) ,$ 
\begin{equation*}
\Uparrow _{\tilde{p}}^{\pi ^{-1}\left( q\right) }\equiv \left\{ \left. 
\tilde{v}\in \nu _{\tilde{p}}^{1}\left( \pi ^{-1}\left( p\right) \right) 
\text{ }\right\vert \text{ }\gamma _{\tilde{v}}\text{ is a segment from }%
\tilde{p}\text{ to }\pi ^{-1}\left( q\right) \right\}
\end{equation*}%
is a $\frac{\pi }{2}$--net in $\nu _{\tilde{p}}^{1}\left( \pi ^{-1}\left(
p\right) \right) .$ Let $W$ be any $k$--dimensional subspace of $T_{\tilde{p}%
}\pi ^{-1}\left( p\right) .$ For $\tilde{v}\in \nu _{\tilde{p}}^{1}\left(
\pi ^{-1}\left( p\right) \right) $ and $\left\{ E_{i}\right\} _{i=1}^{k}$ an
orthonormal basis for $W,$%
\begin{equation*}
\mathrm{Trace}\left( \left. S_{\tilde{v}}\right\vert _{W}\right)
=\left\langle -\displaystyle\sum\limits_{i=1}^{k}\mathrm{II}\left(
E_{i},E_{i}\right) ,\tilde{v}\right\rangle .
\end{equation*}%
Since $\Uparrow _{\tilde{p}}^{\pi ^{-1}\left( q\right) }$ is a $\frac{\pi }{2%
}$--net in $\nu _{\tilde{p}}^{1}\left( \pi ^{-1}\left( p\right) \right) ,$
it follows that for some $\tilde{v}\in \Uparrow _{\tilde{p}}^{\pi
^{-1}\left( q\right) },$ 
\begin{equation}
\mathrm{Trace}\left( \left. S_{\tilde{v}}\right\vert _{W}\right) \leq 0\text{%
. \label{shape gamma inequal}}
\end{equation}

Let $\Lambda _{\pi ^{-1}\left( q\right) }$ be the Lagrangian family of
Jacobi fields along $\gamma _{\tilde{v}}$ that correspond to variations by
geodesics that leave $\pi ^{-1}\left( x\right) $ orthogonally at time $0.$
Then Inequality \eqref{shape gamma inequal} combined with Lemma \ref{sing
soon Ric_k Lemma} gives us that $\gamma _{\tilde{v}}$ has a focal point in $%
\left[ 0,\frac{\pi }{2}\right] .$ As before, it follows that either $\pi
\circ \gamma _{\tilde{v}}$ does not extend to an interval that properly
contains $\left[ 0,\frac{\pi }{2}\right] ,$ or $\pi \circ \gamma _{\tilde{v}%
} $ has a variational conjugate point in $\left[ 0,\frac{\pi }{2}\right] .$
Since $\pi \circ \gamma _{\tilde{v}}$ is a minimal geodesic from $p$ to $q$
and $\mathrm{dist}\left( p,q\right) >\frac{\pi }{2},$ the former case is
excluded. The latter case implies, via Proposition \ref{conj not min prop},
that for all $\varepsilon >0$, $\pi \circ \gamma _{\tilde{v}}|_{\left[ 0,%
\frac{\pi }{2}+\varepsilon \right] }$ is not minimal, so it is also contrary
to our hypothesis that $\mathrm{dist}\left( p,q\right) >\frac{\pi }{2}.$
\end{proof}

\begin{remarknonum}
By Theorem 1 of \cite{ProWilh}, $X$ need not have positive Ricci curvature,
even when $\pi $ is a Riemannian submersion. So neither the first nor third
conclusion of Theorem \ref{submetry cor} follow from the Bonnet-Myers
Theorem.
\end{remarknonum}

\begin{remarknonum}
The sense in which $\gamma $ has a conjugate point can be described via
variations. (See Definition \ref{variat conj dfn} below.) There are also
various notions of conjugacy in length spaces proposed by Shankar and
Sormani in \cite{ShankSorm}. Our variational notion is more readily
adaptable to the situation of Theorem \ref{submetry cor} than are any of
those in \cite{ShankSorm}. All of the definitions have the common feature
that $\gamma $ stops minimizing after a conjugate point.
\end{remarknonum}

\begin{remarknonum}
By results in \cite{GrvGrm1} and \cite{Wilk}, the possibilities for $\pi $
in Part 2 of Theorem \ref{submetry cor} can be listed, if $\pi $ is a
Riemannian submersion. More generally, Riemannian foliations of round
spheres are classified if they are either nonsingular (\cite{LyWilk}) or
they are singular and have fiber dimension $\leq 3$ (\cite{Rad2}). However,
the singular Riemannian foliations of round spheres have not been
classified, and there is an abundance of examples (\cite{Rad}).
\end{remarknonum}

The version of Theorem \ref{submetry cor} when $k=1$ yields, via a
different proof, the inequality statements of Chen and Grove in Theorems A
and B in \cite{ChenGrv}, with the additional information about the behavior
of geodesics from conclusion 1. In particular, if $\pi $ is a Riemannian
submersion, it follows that the conjugate radius of $X$ is $\leq \frac{\pi }{%
2}.$ For a Riemannian submersion $\pi :M^{n+k}\longrightarrow B^{n}$ with
the sectional curvature of $M$ $\geq 1,$ Theorem A of \cite{GonGui} gives
that 
\begin{equation*}
\pi \frac{n-1}{k+n-1}\geq \mathrm{conj}\left( B\right) .
\end{equation*}%
In particular, the conjugate radius of $B$ is $\leq \frac{\pi }{2}$ if $%
k\geq n-1.$ (Cf also Corollary 1.2 of \cite{GonGui}.)

Theorem \ref{submetry cor} has the following consequence for cohomogeneity
one actions.

\begin{corollary}
\label{cohom 1 cor}Let $M$ be a complete Riemannian manifold $M$ with $\Ric%
_{k}\geq k$. If $G\times M\to M$ is an isometric, cohomogeneity one action
with a singular orbit of dimension $\geq k,$ then the following hold.

\noindent 1. The diameter of $M/G$ is $\leq \frac{\pi }{2}.$

\noindent 2. If the diameter of $M/G$ is $\frac{\pi }{2},$ then the
universal cover of $M$ is isometric to the sphere or a projective space with
the standard metrics, and the singular orbits are totally geodesic in $M.$
\end{corollary}

Part 3 of Theorem \ref{submetry cor} has the following corollary.

\begin{corollary}
\label{coh sing orb cor}Let $M$ be a complete Riemannian manifold $M$ with $%
\Ric_{k}\geq k$. If $G\times M\to M$ is an isometric group action, the
diameter of $M/G$ is $>\frac{\pi }{2},$ and $x$ is a point that realizes the
diameter of $M/G,$ then $\dim \pi ^{-1}\left( x\right) $ $\leq k-1.$

In particular, if $G\times M\longrightarrow M$ is as above and is also a
cohomogeneity one action$,$ then both singular orbits have dimension $\leq
k-1.$
\end{corollary}

\begin{remarknonum}
The sectional curvature case of Corollary \ref{coh sing orb cor} can be
inferred from Corollary 2.7 of \cite{ChenGrv}.
\end{remarknonum}

Examples $D$ and $E$ of \cite{GuijWilh} show that the hypothesis about the
dimensions of the submanifolds can not be removed from Theorem \ref{submetry
cor}.

\end{document}